\documentclass[a4paper, 12pt]{article}

\usepackage[OT2,T1]{fontenc}

\usepackage{a4,amsmath,amssymb,amsthm,latexsym,paralist,mathabx,geometry}
\usepackage[english]{babel}
\usepackage[utf8x]{inputenc}
\usepackage{url,xspace,euscript}

\usepackage[arrow,matrix]{xy}
\xyoption{all}
\usepackage{mathabx}
\usepackage{mathtext}
\usepackage{amsmath,amscd}

\newtheorem{theo}{Theorem}[section]
\newtheorem{lemm}[theo]%
 {Lemma}
{Definition}
\newtheorem{prop}[theo]%
{Proposition}
\newtheorem{coro}[theo]%
{Corollary}
{Definition-Proposition}
{Conjecture}
{Remark}
{Question}
{Example}

\newcommand{\finpreuve}{\mbox{} \hfill \mbox{$\Box$}}
\newenvironment{preuve}{\noindent {\it Proof: }}{\finpreuve}

\DeclareSymbolFont{cyrletters}{OT2}{wncyr}{m}{n}

\DeclareMathSymbol{\Sha}{\mathalpha}{cyrletters}{"58}

\newcommand{\Pic}{\mathrm{Pic}}
\newcommand{\Hom}{\mathrm{Hom}}
\newcommand{\cd}{\operatorname{\mathit{cd}}}
\newcommand{\et}{\mathit{et}}

\def\F{{\mathbb F}}

\def\Z{{\mathbb Z}}
\def\fq{{\mathbb F}}

\def\Gm{{\mathbb G}}

\def\Frob{{\rm Frob}}

\def\Gl{{\rm Gl}}

\def\Gal{{\rm Gal}}
\def\ab{{\mathit{ab}}}

\begin{document}

\date{}
\title{The $K(\pi,1)$-property for marked curves over finite fields}

\author{Philippe Lebacque and Alexander Schmidt}

\maketitle

\begin{abstract}
We investigate the $K(\pi,1)$-property for $p$ of smooth, marked curves $(X,T)$ defined over finite fields of characteristic $p$.  We prove that $(X,T)$ has the $K(\pi,1)$-property  if $X$ is affine and give positive and negative examples in the proper case. We also consider the unmarked proper case over a finite field of characteristic different to $p$.\\
2010 Math. Subj. Class. 11R34, 11R37, 14F20
\end{abstract}

\section{Introduction}

In \cite{Sch1},\cite{Sch2},\cite{Sch3}, the second author investigated the $K(\pi,1)$-property for $p$ of arithmetic curves whose function field is of characteristic different to $p$. As  a result, the Galois group  of the maximal unramified outside $S$ and $T$-split pro-$p$-extension of a global field of characteristic different to $p$ is often of cohomological dimension less or equal to two.
In this paper we consider the case of a smooth curve over a finite field of characteristic $p$.
We prove that $(X,T)$ has the $K(\pi,1)$-property  if $X$ is affine and give positive and negative examples in the proper case. We also consider the unmarked proper case over a finite field of characteristic different to $p$, which was left out in the earlier papers.

\medskip\noindent
The authors would like to thank the referee for his valuable suggestions.

\subsection{The marked \'{e}tale site and the \boldmath $K(\pi,1)$-property}
Let $X$ be a regular one-dimensional noetherian  scheme defined over $\F_{q}$ (with $q=p^f$) and let $T$ be a finite set of closed points. In \cite{Sch3}, the second author defined the marked site $(X,T)$ of $X$ at $T$ considering finite \'etale  morphisms $Y\to X$ inducing isomorphisms $k(y)\cong k(x)$ on the residue fields for any closed point $y\in Y$ mapping to $x\in T$. Let $M$ be a $p$-torsion sheaf. The resulting cohomology groups are denoted by $H^i(X,T,M)$ and they satisfy the usual properties we expect from \'etale cohomology groups. He also proved (see \cite{Sch3} for more details) that these finite marked \'{e}tale morphisms fit into a Galois theory and (after choosing a base geometric point $\widebar{x}\notin T$) we denote by $\pi_{1}(X,T)$ the profinite group classifying \'{e}tale coverings of $X$ in which the points of $T$ split completely. We denote by $\widetilde{(X,T)}(p)$ the universal pro-$p$-covering of $(X,T)$. The projection $\widetilde{(X,T)}(p)\to X$ is Galois with Galois group the maximal pro-$p$-quotient $\pi_{1}(X,T)(p)$ of  $\pi_{1}(X,T)$.

Let $M$ be a discrete $p$-torsion $\pi_{1}(X,T)(p)$-module. Consider the Hochschild-Serre spectral sequence:
\[
E_2^{ij}=H^i(\pi_{1}(X,T)(p),H^j(\widetilde{(X,T)}(p),T,M))\Rightarrow H^{i+j}(X,T,M).
\]
The edge morphisms provide homomorphisms
\[
\phi_{i,M}:H^i(\pi_{1}(X,T)(p),M)\to H^{i}(X,T,M).
\]
We say that $(X,T)$ has the $K(\pi,1)$-property for $p$ if $\phi_{i,M}$ is an isomorphism for all $M$ and all $i\ge 0$.  The following Lemma~\ref{kpi} implies in particular, that $(X,T)$ has the $K(\pi,1)$-property for $p$ if $\phi_{i,\F_p}$ is an isomorphism for $i\ge 2$.
\begin{lemm}(cf. \cite{Sch3} Lemma 2.2) \label{kpi} $\phi_{i,M}$
is an isomorphism for $i=0,1$ and is a monomorphism for $i=2$. Moreover, $\phi_{i,M}$ is an isomorphism for all $i\geq 0$ if and only if
\[
\varinjlim_{(Y,T')} H^i(Y,T',M)=0 \ \text{for all }i\geq 1,
\]
where the direct limit is taken over all finite intermediate coverings $(Y,T')$ of the universal pro-$p$-covering  $\widetilde{(X,T)}(p)\to (X,T).$
\end{lemm}

\subsection{Notation} Unless otherwise stated, we use the following notation:
\begin{compactitem}
\item[-] $p$ denotes a prime number.
\item[-] $\F$ is a finite field, $\widebar{\F}$ an algebraic closure of $\F,$ $\widetilde{\F}$ its maximal pro-$p$-extension inside $\widebar{\F}$ and $G_{\F}$ the Galois group of $\widebar{\F}/\F$.
\item[-] $X$ is a smooth projective absolutely irreducible curve defined over $\F$.
\item[-] $k=\F(X)$ the function field of $X$.
\item[-] $g_X$ the genus of $X$.
\item[-] $\widebar{X}=X\times_{\F}\widebar{\F}$, $\widetilde{X}=X\times_{\F}\widetilde{\F}$.
\item[-] $S$, $T$ are two disjoint sets (possibly empty) of closed points of $X$.
\item[-] if $x$ is a closed point of a $X,$ $X_{x}$ denotes the henselization of $X$ at $x$ and $T_{x}=\{x\}$ if $x\in T$ and $\emptyset$ otherwise.
\item[-]  $k_{S}^T$ denotes the maximal pro-$p$-extension of $k$ which is unramified outside $S$ and  in which all places of $T$  split completely. If empty, we omit $S$ (or $T$) from the notation.
\item[-] $G_S^T(k)=\mathrm{Gal}(k_S^T/k)=\pi_1(X-S,T)(p)$.
\item[-] $H^i(X-S,T)$ denotes the $i$-th \'etale cohomology group $H^i_\et(X-S,T,\F_{p})$ of the marked curve $(X-S,T)$.
\item[-] for a pro-$p$-group $G$ we set $H^i(G)=H^i(G,\fq_p)$.
\item[-] for an abelian group $A$ and an integer $m$ we write $A[m]= \ker (A\stackrel{\cdot m}{\to} A)$
\end{compactitem}

\subsection{New results}
Let $X$ be a smooth projective absolutely irreducible curve defined over the finite field $\F$ and let $k=\F(X)$ be the function field of $X$. Let $S$ and $T$ be finite disjoint sets of closed points of $X$. In this paper, we prove the following result:

\begin{theo}\label{main} Assume that $p=\mathrm{char}(\F)$.\smallskip
\begin{compactitem}
\item[\rm (i)] If $S\neq \emptyset$, then $(X-S,T)$ has the $K(\pi,1)$-property for $p$ and $\cd G_S^T(k)=1$.
\item[\rm (ii)] If $T=\emptyset$, then $(X-S)$ has the $K(\pi,1)$-property for $p$ and $\cd G_S(k)\leq 2$.
\end{compactitem}
\end{theo}
\setdefaultleftmargin{3em}{1em}{}{}{}{}
In the remaining cases, we have the following results.
\begin{theo}\label{th2} Assume that $p=\mathrm{char}(\F)$,  $S=\emptyset$ and $T\ne \emptyset$.\smallskip
\begin{compactitem}
\item[\rm (i)] If\/ $\Pic(X)[p]= 0$, then $(X,T)$ has the $K(\pi,1)$-property for $p$ if and only if $T=\{x\}$ consists of a single point with $p\nmid \deg x$. In this case $\pi_1(X,T)(p)=1$.
\item[\rm (ii)] If\/ $\Pic(X)[p]\neq 0$ and
\[
\sum_{x\in T} \frac{\deg (x)}{(\#\F) ^{\deg (x)/2}-1} > g_X-1,
\]
then $\pi_1(X,T)(p)$ is finite and $(X,T)$ has not the $K(\pi,1)$-property for $p$.
\end{compactitem}
\end{theo}

Finally, we consider the unmarked proper case over a finite field of characteristic different to $p$, which was left out in the earlier papers.

\begin{theo}\label{uglp} Assume that $p\neq \mathrm{char}(\F)$.
Then $X$ has the $K(\pi,1)$-property for $p$ if and only if  $\mu_p (\F)=1$ or $\Pic(X)[p]\neq 0$.

In the remaining case $\mu_p\subset \F$ and $\Pic(X)[p]= 0$ we have
\[
\pi_1^\et(X)(p)\cong \pi_1^\et(\F)(p)\cong \Z_p.
\]
In particular, $H^i(\pi_1^\et(X)(p))$ is always finite and vanishes for $i>3$.
\end{theo}

\section{Computation of  \'{e}tale cohomology groups}

\begin{prop}[Local computation]\label{loc}
 Let $K$ be a nonarchimedean local (or henselian) field of characteristic $p$.  Let $Y=Spec\,\mathcal{O}_K$, $y\in Y$  the closed point and let $T$  be $\emptyset$ or $\{y\}$. Then the local cohomology groups $H^i_y(Y,T)$ vanish for $i\neq 2$ and
\[
H_y^2(Y,T)=\begin{cases}
H^1_{/nr}(K)\text{ if }T=\emptyset\\
H^1(K)\ \text{ if } T=\{y\},                                                                                                                                                                                                                                                                                                                                                                                                                                                                           \end{cases}
\]
where $H^1_{/nr}=H^1(K)/H^1_{nr}(K)$.
\end{prop}
\begin{preuve}
We use the excision sequence:
\[
\cdots\to H_y^i(Y,T)\to H^i(Y,T)\to H^i(Y-\{y\})\to H_y^{i+1}(Y,T)\to \cdots.
\]
Since $Y$ is henselian, $H^i(Y)\cong H^i(y)=H_{nr}^i(K)$, hence  $H^i(Y)=0$ for $i\geq 2$. Since $Y$ is normal, $H^1(Y,T)\to H^1(Y-\{y\})$ is injective, hence $H_y^1(Y,T)=0$. Furthermore, $H^i(Y-\{y\})=H^i(K)$ and this group vanishes for $i\ge 2$ since $cd_pK=1$ (see \cite{NSW}, Cor.~6.1.3). It follows that $H^i_{y}(Y,T)\cong H^i(Y,T)$ for $i\geq 3$.

For $T=\emptyset$ we obtain $H^i_{y}(Y)=0$ for $i\geq 3$ and the short exact sequence
\[
0\to H^1(Y)\to H^1(Y-\{y\})\to H_y^{2}(Y)\to 0
\]
implies the result for $H^2_y(Y)$.

If $T=\{y\}$, the identity of $(Y,T)$ is cofinal among the covering families of $(Y,T)$, hence $H^i(Y,\{y\})=0$ for $i\geq 1$.
We obtain $H^2_{y}(Y,\{y\})\cong H^1(K)$ and $H^i_y(Y,\{y\})=0$ for $i\geq 3$.
\end{preuve}

\begin{prop}(Global computation)\label{GC}
Let $X$ be a smooth projective  and geometrically irreducible curve over $\F$,  $k=\F(X)$ and  $S$ and $T$ finite, disjoint sets of closed points of~$X$.

\medskip
Then $H^i(X-S,T)=0$ for $i\geq 3$ and $H^2(X-S,T)= 0$ if $S\ne \emptyset$.
We have an exact sequence
\[
0\to H^1(X-S,T)\to H^1(X-S)\to \bigoplus_{x\in T}H^1_{nr}(k_x)\to H^2(X-S,T)\to H^2(X-S)\to 0.
\]
\end{prop}

\begin{preuve}
In the case $T=\emptyset$ we have $H^i(X-S)=0$ for $i\geq 3$ and $H^2(X-S)=0$ if $S\neq \emptyset$ by \cite{SGA4.3} exp.\ 10, Thm.\ 5.1 and  Cor.\ 5.2. Moreover, the sequence is exact for trivial reasons.

Now assume $T\neq \emptyset$.
Consider the excision sequence for $(X-S,T)$ and $(X-(S\cup T))$:
\[
\dots\to \bigoplus_{x\in T}H_x^i((X-S)_x,T_x)\to H^i(X-S,T)\to H^i(X-(S\cup T))\to\cdots .
\]
Proposition \ref{loc} shows that $H^i(X-S,T)\cong H^i(X-(S\cup T))=0$ for $i\geq 3$ and the exactness of the sequence
\[
0\to H^1(X-S,T)\to H^1(X-(S\cup T))\to \bigoplus_{x\in T}H^1(k_x)\to H^2(X-S,T)\to 0.\eqno (*)
\]
\medskip
Comparing this with the excision sequence for $(X-S)$ and $(X-(S\cup T))$
\[
0\to H^1(X-S)\to H^1(X-(S\cup T))\to \bigoplus_{x\in T}H^1_{/nr}(k_x)\to H^2(X-S)\to 0,
\]
we obtain the exact sequence of the proposition.

If $S\neq \emptyset$, the Strong Approximation Theorem implies that
\[
H^1(X-(S\cup T))\longrightarrow \bigoplus_{x\in T}H^1(k_x)
\]
is surjective (see \cite{NSW} Thm.~9.2.5). Using $(*)$ this shows that $H^2(X-S,T)=0$ in this case.
\end{preuve}

\begin{coro} \label{notkpi1}
If $G_S^T(k)$ is finite and nontrivial, then $(X-S,T)$ does not have the $K(\pi,1)$-property for $p$.
\end{coro}

\begin{proof}
In this case we have $\cd  G_S^T(k)=\infty$ but $H^i(X-S,T)=0$ for $i\ge 3$.
\end{proof}

\begin{coro} \label{EP} We have the Euler-Poincar\'{e} characteristic formula
\[
\sum_{i=0}^2 (-1)^i \dim_{\F_p} H^i(X,T)= \#T.
\]
\end{coro}

\begin{proof} If $S=\emptyset$, all groups in the  exact sequence of Proposition \ref{GC} are finite and we obtain
\[
\sum_{i=0}^2 (-1)^i \dim_{\F_p} H^i(X,T) = \# T + \sum_{i=0}^2 (-1)^i \dim_{\F_p} H^i(X).
\]
Recall that $H^1(\widebar X)=\Hom(\Pic(\widebar X)[p],\F_p)$ (every connected \'{e}tale covering of $\widebar X$ comes by base change from an isogeny of the Jacobian of $\widebar X$). Hence
\[
H^2(X)= H^1(\F,H^1(\widebar X))= H^1(\F,\Hom(\Pic(\widebar X)[p],\F_p)) =\Hom(\Pic(\widebar X)[p],\F_p)_{G_\F}.
\]
Furthermore, we have an exact sequence
\[
0\to H^1(\F) \to H^1(X) \to \Hom(\Pic(\widebar X)[p], \F_p)^{G_\F} \to 0.
\]
Thus Lemma~\ref{invconv} below shows $$\sum_{i=0}^2 (-1)^i \dim_{\F_p} H^i(X)=1-\dim_{\F_p} H^1(\F)=0.$$
\end{proof}
\begin{lemm}\label{invconv}
We have $\Pic(\widebar X)[p]^{G_\F}= \Pic(X)[p]$ and
\[
\dim_{\F_p} \Pic(\widebar X)[p]_{G_\F}= \dim_{\F_p}\Pic(X)[p].
\]
\end{lemm}
\begin{proof}
The first equality follows from the Leray spectral sequence
\[
E_2^{ij}=H^i(\F,H^j(\widebar X,\Gm_m)) \Rightarrow H^{i+j} (X,\Gm_m)
\] and the vanishing of the Brauer group of a finite field:
\[
H^2(\F, H^0(\widebar X,\Gm_m))= H^2(\F, \widebar \F^\times)=0.
\] The equality of dimensions  follows from the exact sequence of finite-dimensional $\F_p$-vector spaces
\[
0 \to  \Pic(\widebar X)[p]^{G_\F} \to \Pic(\widebar X)[p] \stackrel{1-\Frob}{\longrightarrow}  \Pic(\widebar X)[p] \to
\Pic(\widebar X)[p]_{G_\F} \to 0.
\]
\end{proof}

\section{Proof of Theorem~\ref{main}}
Assume $S\neq \emptyset$. From the computations in the last section, we know that $H^i(X-S,T)=0$ for $i\ge 2$. By Lemma~\ref{kpi}, $(X-S,T)$ has the $K(\pi,1)$-property for~$p$ and $\cd G_S^T(k)\le 1$.
But $G_S^T(k)$ is nontrivial, which follows from the exact sequence
\[
0\to H^1(X-S,T)\to H^1(X-S)\to \bigoplus_{x\in T}H^1_{nr}(k_x)\to 0
\]
together with the fact that $\bigoplus_{x\in T}H^1_{nr}(k_x)$ has finite $\F_p$-dimension whereas $H^1(X-S)$ is infinite dimensional.

\bigskip
Now assume that $S=\emptyset$ and $T=\emptyset$.  Let $\widetilde \F$ be the maximal $p$-extension of $\F$ in $\widebar \F$. Then $H^2(X_{\widetilde \F})=H^2(X_{\widebar \F})^{\Gal(\widebar \F/\widetilde \F)}=0$. Hence $X_{\widetilde \F}$ is a $K(\pi,1)$ for $p$ and the Hochschild-Serre spectral sequence for $X_{\widetilde \F}/X$ shows the same for $X$.
This finishes the proof of Theorem~\ref{main}.

\section{Proof of Theorem \ref{th2}}

\begin{prop}\label{finite}
Assume that $\Pic(X)[p]=0$ and $T\ne \emptyset$ and let $p^r$ be the maximal $p$-power dividing $gcd(\deg x,\ x\in T)$. Then
\[
G^T(k)=\pi_1(X,T)(p)\cong\Gal(\F'/\F),
\]
where $\F'$ is the unique extension of\/ $\F$ of degree $p^r$.
\end{prop}

\begin{proof} Let $\widetilde \F$ be the maximal $p$-extension of $\F$ in $\widebar \F$. Using Lemma~\ref{invconv}, we have
\[
H^2(X)=H^1(\F,H^1(\widebar X))\cong \Hom(\Pic(X)[p],\F_p)=0
\]
and Corollary \ref{EP} shows that $H^1(X)$ is $1$-dimensional. Hence $\pi_1(X)(p)$ is free of rank~$1$ and therefore the surjection
\[
\pi_1(X)(p) \twoheadrightarrow \Gal(\widetilde \F/\F)
\]
is an isomorphism (cf.\ \cite{NSW}, Prop.~1.6.15). The maximal subextension $\F'/\F$ of $\widetilde \F /\F$ such that all points in $T$ split completely in the base change $X\otimes_{\F}\F' \to X$ is exactly the unique extension of degree~$p^r$ of $\F$.
\end{proof}

\begin{coro}
Assume that $\Pic(X)[p]=0$ and $T\ne \emptyset$. Then $(X,T)$ is a $K(\pi,1)$ for $p$ if and only if\/ $T=\{x\}$ consists of a single point with $p\nmid \deg x$. In this case the fundamental group $\pi_1(X,T)(p)$ is trivial.
\end{coro}

\begin{proof}
By Proposition~\ref{finite}, $\pi_1(X,T)(p)$ is finite cyclic. If $p\mid gcd(\deg x,\ x \in T)$, then  $\pi_1(X,T)(p)$  is nontrivial and $(X,T)$ is not a $K(\pi,1)$ for $p$ by Corollary \ref{notkpi1}.

Assume  $p\nmid gcd(\deg x \mid x \in T)$. Then $\pi_1(X,T)(p)$ is the trivial group, $H^1(X,T)=0$ and $(X,T)$ is  a $K(\pi,1)$ if and only if $H^2(X,T)=0$. By Corollary~\ref{EP} this is equivalent to $\# T=1$.
\end{proof}

\begin{lemm} Assume that $\pi_1(X,T)(p)$ is finite and $\Pic(X)[p]\neq 0$. Then $(X,T)$ is not a $K(\pi,1)$ for $p$.
\end{lemm}

\begin{proof}
By Corollary \ref{notkpi1}, $(X,T)$ is not a $K(\pi,1)$ for $p$ if $\pi_1(X,T)(p)$ is nontrivial. Assume that $\pi_1(X,T)(p)=1$. Then $(X,T)$ is a $K(\pi,1)$ for $p$ if and only if $H^2(X,T)=0$. But by Proposition~\ref{GC}, $H^2(X)\cong\Hom(\Pic(X)[p],\F_p)\neq 0$ is a quotient of $H^2(X,T)$. \end{proof}

The following theorem is due to Ihara, see \cite{Ih}, Thm.~1 (FF).

\begin{theo}
Assume that $T\ne\varnothing$ and let $q=\# \F$. If
\[
\sum_{x\in T} \frac{\deg (x)}{q^{\deg (x)/2}-1} > \max(g_X-1,0),
\]
then $\pi_1(X,T)$ is finite. In particular, $\pi_1(X,T)(p)$ is finite.
\end{theo}

Summing up, we obtain Theorem~\ref{th2}.

\section{Proof of Theorem~\ref{uglp}}

Let $\widetilde \F$ be the maximal $p$-extension of $\F$ in $\widebar \F$ and  $\widetilde X=X \times_\F \widetilde \F$. Then
\[
X \text{ is a } K(\pi,1) \text{ for } p \Longleftrightarrow \tilde X \text{ is a } K(\pi,1) \text{ for } p
\]
and we have
\[
H^i_\et(\widetilde X) \cong H^i_\et(\widebar X)^{G(\widebar \F / \widetilde \F)}
\]
for all $i$. Hence $H^i_\et(\widetilde X)$ vanishes for $i\geq 3$ and $H^2_\et(\widetilde X)=\mu_p(\widetilde \F)^*= \mu_p(\F)^*$.
We conclude that $X$ has the  $K(\pi,1)$-property for $p$ if $\mu_p(\F)=1$. In the following we assume that $\F$ contains all $p$-th roots of unity.   For every tower of finite connected \'{e}tale $p$-coverings $Z\to  Y\to X$  the natural map
\[
 \Z/p\Z  = H^2_\et(\widetilde Y,\mu_p) \longrightarrow H^2_\et(\widetilde Z,\mu_p)=\Z/p\Z
\]
is multiplication by the degree  $[\widetilde Z:\widetilde Y]$. Hence, by Lemma~\ref{kpi},
\[
\widetilde X \text{ is a } K(\pi,1) \text{ for } p \Longleftrightarrow \# \big(\pi_1^\et(\widetilde X)(p)\big)= \infty.
\]
Note that
\[
\begin{array}{rcl}
\pi_1^\ab(\widetilde X)/p& \cong & H^1_\et(\widetilde X)^*\\
&\cong & \big(H^1_\et(\widebar X)^{G(\widebar \F /\widetilde \F)})^* \\
&\cong& \Pic(\widebar X)[p]_{G(\widebar \F /\widetilde \F)}
\end{array}
\]
and by Lemma~\ref{invconv}
\[
\Pic(\widebar X)[p]_{G(\widebar \F /\widetilde \F)} = 0 \Longleftrightarrow  \Pic(\widetilde X)[p]=0.
\]
Furthermore, since $G(\widetilde \F/\F)$ is a pro-$p$-group:
\[
\Pic(\widetilde X)[p]=0 \Longleftrightarrow \Pic(X)[p]= \Pic(\widetilde X)[p]^{G(\widetilde \F/\F)} = 0,
\]
and therefore, $\Pic(X)[p]\neq 0\Longleftrightarrow \pi_1^\ab(\widetilde X)/p\neq 0$.
Hence it suffices to show the equivalences
\[
\# \big(\pi_1^\et(\widetilde X)(p)\big)= \infty \Leftrightarrow \# \big(\pi_1^\ab(\widetilde X)(p)\big)= \infty \Leftrightarrow \# \pi_1^\ab(\widetilde X)/p\neq 0.
\]
Elementary theory of pro-$p$-groups shows that it remains to show the implication
\[
\pi_1^\ab(\tilde X)/p\neq 0 \Longrightarrow \# \big(\pi_1^\ab(\tilde X)(p)\big) =\infty.
\]
Setting $T:=T_p(\widebar X)=\pi_1^\ab(\widebar X)(p)$, we can write this implication in the form
\[
(T_{G(\widebar \F/\widetilde \F)})/p\neq 0 \Longrightarrow \# (T_{G(\widebar \F/\widetilde \F)}) = \infty.
\]
The group $G(\widebar \F / \tilde \F)$ is pro-cyclic of supernatural order prime to $p$. Furthermore, $T\cong \Z_p^{2g}$ and the kernel of the reduction map $\Gl_{2g}(\Z_p)\to \Gl_{2g}(\F_p)$ is a  pro-$p$-group. Hence the action of $G(\widebar \F / \widetilde \F)$ on $T$ factors through a finite cyclic group of order prime to $p$. We conclude that Theorem~\ref{uglp} follows from Lemma~\ref{koinvlem} below. \hfill $\Box$

\bigskip
The following Lemma \ref{koinvlem} and its application in the proof of Theorem~\ref{uglp} were proposed to us by J.~Stix. We thank the referee for suggesting the short proof given below.

\begin{lemm}\label{koinvlem}
Let $G$ be a finite group of order $n$, $p$  a prime number with  $p\nmid n$ and $T$  a finitely generated free $\Z_p$-module with a $G$-action. Then
\[
\# T_G = \infty \Longleftrightarrow (T/p)_G \neq 0.
\]
\end{lemm}

\begin{proof}
Since the Tate cohomology of $\Z_p[G]$-modules vanishes, we obtain the split exact sequence of $\Z_p[G]$-modules
\[
0 \longrightarrow \ker(N) \longrightarrow T \stackrel{N}{\longrightarrow} T^G \longrightarrow 0,
\]
where $N=\sum_{g\in G} g$. For $B=\ker(N)$, $\hat H^{-1}(G,B)=0$ implies $B_G=0$. We obtain $T_G\cong T^G$ and $(T/p)_G\cong (T^G)/p$.
Hence both assertions of the lemma are equivalent to $T^G\ne 0$.
\end{proof}

\bigskip\small
{\sc Laboratoire de Math\'ematiques de Besan\c con, 16 route de Gray, 25030 Besan\c con, France}

{\sc INRIA Saclay - Ile-de-France, Equipe-projet GRACE}

email: {\tt philippe.lebacque@univ-fcomte.fr}

\medskip
{\sc Universit\"{a}t Heidelberg, Mathematisches Institut, Im Neuenheimer Feld 288, D-69120 Heidelberg, Deutschland}

email: {\tt schmidt@mathi.uni-heidelberg.de}
\end{document}